\documentclass[a4paper,11pt]{article}
\usepackage[T1]{fontenc}
\usepackage[utf8]{inputenc}
\usepackage{lmodern}
\usepackage{graphicx}
\usepackage{float}
\usepackage{amsmath}
\usepackage{amssymb}
\usepackage{enumerate}
\usepackage{amsthm}
\usepackage{diagrams}
\usepackage[autolanguage, np]{numprint}

\newtheorem{theorem}{Theorem}[section]
\newtheorem{lemma}[theorem]{Lemma}
\newtheorem{proposition}[theorem]{Proposition}
\newtheorem{corollary}[theorem]{Corollary}

\newcommand{\comments}[1]{}

\newenvironment{example}[1][Example]{\begin{trivlist}
\item[\hskip \labelsep {\bfseries #1}]}{\end{trivlist}}

\DeclareMathOperator{\Ker}{Ker}

\newcommand{\bbZ}{\mathbb{Z}}

\DeclareMathOperator{\Id}{Id}
\DeclareMathOperator{\Sub}{Sub}
\DeclareMathOperator{\Atp}{Atp}
\DeclareMathOperator{\Aut}{Aut}
\DeclareMathOperator{\Hom}{Hom}
\DeclareMathOperator{\B}{B}
\DeclareMathOperator{\C}{C}
\DeclareMathOperator{\HH}{H}

\DeclareMathOperator{\Span}{Span}
\DeclareMathOperator{\Map}{Map}
\DeclareMathOperator{\Pred}{Pred}
\DeclareMathOperator{\Inv}{Inv}

\title{Enumeration of nilpotent loops up to isotopy}
\author{Lucien Clavier\thanks{We would like to take this opportunity to give our warmest thanks to Petr Vojt\v{e}chovsk\'{y} and Dan Daly for their continued interest and helpful comments concerning the present work.}}

\begin{document}
\maketitle
\begin{abstract}
We modify tools introduced in \cite{Petr} to count, for any odd prime~$q$, the number of nilpotent loops of order $2q$ up to isotopy, instead of isomorphy.
\end{abstract}





\section{Introduction}
\label{intro}

Recall that a set $Q$ equipped with a binary operation $\cdot$ is a \textit{loop} if it possesses a neutral
element and if for each $a$, $b$ in $Q$ there exist unique $x$, $y$ such that
\[ a \cdot x=b \text{ and }y \cdot a=b. \]

As usual, we write these respectively as $ x=a \backslash b \text{ and } y=b/a$.
We abbreviate $x \cdot y$ as $xy$, and adopt the usual convention that multiplication should be performed first between contiguous elements, and then between dotted elements.
For instance, $ xy \cdot z$ is the same as $(x \cdot y) \cdot z$.

Recall that groups are exactly associative loops.
Also, normalized latin squares are exactly multiplication tables of finite loops.

The \textit{center} $Z(Q)$ of a loop $Q$ consists of all elements $x$ in $Q$ such that
\[ xy = yx, \; xy \cdot z = x \cdot yz, \;yx \cdot z = y \cdot xz \text{ and } yz \cdot x = y \cdot zx \]
for every $y$, $z$ in $Q$.

\textit{Normal} subloops are kernels of loop homomorphisms. The center
$Z(Q)$ is a normal subloop of $Q$. The \textit{upper central series} $Z_0(Q) \leq Z_1(Q) \leq \ldots $ is defined
inductively by
\[Z_0(Q) = 1, \; Q/Z_{i+1}(Q) = Z(Q/Z_i(Q)).\]
If $Z_{n-1}(Q) < Z_n(Q) = Q$ for some $n$, we say that $Q$ is \textit{(centrally) nilpotent
of class n}.

A triple $t=(\alpha,\beta,\gamma)$ of bijections between two loops $(L_1, \cdot)$ and $(L_2,\ast)$ is an \textit{isotopism} if
\[ \alpha(x) \ast \beta(y) = \gamma(x \cdot y) \]
for each $x$, $y$ in $L_1$. If such a triple exists, $L_1$ and $L_2$ are said to be \textit{isotopic}.
Isotopy defines a relation of equivalence; if two loops are isomorphic, they must be isotopic (it is the case when we can choose $\alpha=\beta=\gamma$). We write $\cong$ for the relation of isomorphy and $\simeq$ for the relation of isotopy.

An \textit{autotopism} of a loop $L$ is an isotopism from $L$ to $L$. We write $\Atp(L)$ for the set of all autotopisms of a loop $L$; it is a group with respect to the law of composition. 

\mbox{}

We believe the present article is more or less self-contained, but we invite the reader to see \cite{Petr} for any shortcut we may have used.
Also, since both articles have the same scheme, most ideas here will appear more natural to those readers that are already well acquainted with \cite{Petr}.

Here is a summary of the paper, with $A$ an abelian group, $F$ a loop.

Section 2.
This section is identical to Section 2 in \cite{Petr}, and was added for the sake of completeness. 
Namely, central extensions of $A$ by $F$ are in one-to-one correspondence with (normalized) cocycles.
If two cocycles differ by a coboundary, their associated loops are isomorphic.

Section 3.
The group $\Atp(F,A)=\Atp(F)\times \Aut(A)$ acts on $\C(F,A)$ via, for $t=(\alpha,\beta,\gamma)$:
\[
(t,h): \theta \mapsto N(h\theta(\alpha^{-1},\beta^{-1}))
\]
where $N$ is the ``normalizing'' projection defined by
\[
N(m)(x, y) = m(x, y) - m(x, 1) - m(1, y) + m(1, 1) .
\]
This induces an action on $\HH(F,A)$; every orbit under this action consists of cocycles whose associated loops are isotopic. 

Section 4.
For a given cocycle $\theta$, if every central extension of $A$ by $F$ isotopic to the loop $Q(F,A,\theta)$ is in the orbit of $\theta$, we say that $\theta$ is separable.
We provide some conditions under which cocycles are separable.

Section 5.
We define (starred) invariant spaces of subgroups of $\Atp(F,A)$ in the same way as in \cite{Petr}.
Therefore, if every cocycle is separable, we can count the number of central extensions of $A$ by $F$ up to isotopy, as soon as we know the subgroup structure of $\Atp(F,A)$ and the cardinality of the starred invariant space of each subgroup of $\Atp(F,A)$.

Section 6.
We study the case where $A=\bbZ_2$, $F=\bbZ_q$ with $q$ an odd prime.
In that case, we know from \cite{Cla} the subgroup structure of $\Atp(F)$ (see Subsection 6.1).
Thus, we only have left to compute the invariant (resp. starred invariant) spaces of such subgroups. This is done in Subsection 6.2 (resp. 6.3).

Subsequently, we can compute the number $\widetilde{\mathcal{N}}(2q)$ of nilpotent loops of order $2q$ up to isotopy (Theorem 6.10), and describe the asymptotic growth of $\widetilde{\mathcal{N}}(2q)$ (Corollary 6.11).

Section 7. 
We provide some ideas related to the present work. See also Section 10 in \cite{Petr}.


\section{Central extensions, cocycles and coboundaries}
\label{centr}
Let $A$ be an abelian group and $F$ a loop.
A loop $Q$ is a \textit{central extension of $A$ by $F$} if $A \leq Z(Q)$ and $Q/A \cong F$.

A mapping $\theta : F \times F \rightarrow A$ is a \textit{(normalized) cocycle} if it satisfies for every $x \in F$
\[ 
\theta(1, x) = \theta(x, 1) = 0 .
\]
For a cocycle $\theta$, define $Q(F, A, \theta)$ to be $F \times A$ equipped with the multiplication:
\[
(x, a)(y, b) = \left(xy, a + b + \theta(x, y)\right).
\]

The following characterization of central loop extensions is folklore, and is in
complete analogy with the associative case:
\begin{theorem} The loop $Q$ is a central extension of $A$ by $F$ if and only if there is a
cocycle $\theta$ such that $Q \cong Q(F, A, \theta)$.
\end{theorem}
The cocycles form an abelian group $\C(F,A)$ with respect to the natural addition; when $A$ is a field, $\C(F,A)$ is a vector space over $A$ with the natural scalar multiplication. 

Define
\begin{align*}
&\Map_0(F,A) = \{ \tau : F \rightarrow A;\; \tau (1) = 0\},  \\ 
&\Hom(F,A) = \{\tau : F \rightarrow A ;\; \tau \text{ is a homomorphism of loops}\}.
\end{align*}

\begin{lemma}
The mapping $\, \widehat{\text{ }} : \Map_0(F,A) \rightarrow \C(F,A), \tau \mapsto \widehat{\tau}$ defined by
\[
\widehat{\tau}(x, y) = \tau (xy) - \tau (x) -\tau (y)
\]
is a group homomorphism with kernel $\Hom(F,A)$.
\end{lemma}

The image 
\[
\B(F,A) = \widehat{\C(F,A)} \cong \Map_0(F,A)/\Hom(F,A)
\]
is a subgroup (subspace) of $\C(F,A)$; its elements are referred to as \textit{coboundaries}.
Coboundaries play a prominent role in classifications due to this simple observation:
\begin{lemma}
\label{isomorphic}
Let $\widehat{\tau} \in \B(F,A)$. Then $f : Q(F, A, \theta) \rightarrow Q(F, A, \theta + \widehat{\tau})$ defined by
\[f(x, a) = (x, a + \tau(x))\]
is an isomorphism of loops.
\end{lemma}
Thus, it is sufficient to consider cocycles modulo coboundaries, and we define the \textit{second cohomology} 
\[\HH(F,A)=\C(F,A)/\B(F,A).\]


\section{Action of Autotopism groups}
\label{action}

Following \cite{Petr}, we are going to define an action of $\Atp(F,A)$ on $\C(F,A)$ and $\HH(F,A)$.
For any cocycle $\theta$ and any autotopism $t=(\alpha, \beta, \gamma)$ of $F$, we would like to define something like the map 
\[
(x,y) \mapsto \theta(\alpha^{-1}(x),\beta^{-1}(y))
\]
but this is usually not a normalized cocycle. 

Instead, let $N$ be the function defined for any $m : F \times F \rightarrow A$ by
\[
N(m)(x, y) = m(x, y) - m(x, 1) - m(1, y) + m(1, 1).
\]

Notice that $N(m)$ is always a cocycle, and that $N$ restricted to $\C(F,A)$ is the
identity map; thus, when $A$ is a field, $N$ is a projection from $\Map(F \times F, A)$ onto $\C(F,A)$.

Now, let 
\[
\Atp(F,A) = \Atp(F) \times \Aut(A).
\]
Write for every $t = (\alpha, \beta, \gamma) \in \Atp(F) $ and every $h \in \Aut(A)$
\[
 ^{(t,h)} \theta = N(h\theta(\alpha^{-1},\beta^{-1})).
\]
By convention, $\theta(\alpha^{-1},\beta^{-1})$ stands for the element of $\Map(F \times F, A)$ defined by $\theta(\alpha^{-1},\beta^{-1})(x,y)=\theta(\alpha^{-1}x,\beta^{-1}y)$.

\begin{lemma}
The group $\Atp(F,A)$ acts on $\C(F,A)$ via
\[
(t,h)\cdot \theta = ^{(t,h)} \theta
\]
\end{lemma}
\begin{proof}
The proof is straightforward. Nevertheless, we would like to prove associativity here, considering the following computation to be non-trivial from the formal point of view.
For all $(t_1,h_1),(t_2,h_2) \in \Atp(F,A)$, $\theta \in \C(F,A)$ and $x,y\in F$,
$^{(t_1,h_1)}\left(^{(t_2,h_2)} \theta\right)(x,y)$ decomposes into 16 terms.
Namely, it equals after unpacking $^{(t_1,h_1)}\left(^{(t_2,h_2)} \theta\right)$ into $^{(t_1,h_1)}\left((x,y)\mapsto \left( ^{(t_2,h_2)} \theta\right) (x,y)\right)$: 
\begin{small}
\[ \begin{array}{cccc}
   & h_1 h_2 \theta\left(\alpha_2^{-1} \alpha_1^{-1}\left(x\right),\beta_2^{-1} \beta_1^{-1}\left(y\right)\right)	& 
- & h_1 h_2 \theta\left(\alpha_2^{-1} \alpha_1^{-1}\left(x\right),\beta_2^{-1}\left(1\right)\right) \\
 - & h_1 h_2 \theta\left(\alpha_2^{-1}\left(1\right),\beta_2^{-1} \beta_1^{-1}\left(y\right)\right) & 
+ & h_1 h_2 \theta\left(\alpha_2^{-1}\left(1\right),\beta_2^{-1}\left(1\right)\right)   \\

-  & h_1 h_2 \theta\left(\alpha_2^{-1} \alpha_1^{-1}\left(x\right),\beta_2^{-1} \beta_1^{-1}\left(1\right)\right)	& 
+ & h_1 h_2 \theta\left(\alpha_2^{-1} \alpha_1^{-1}\left(x\right),\beta_2^{-1}\left(1\right)\right) \\
 + & h_1 h_2 \theta\left(\alpha_2^{-1}\left(1\right),\beta_2^{-1} \beta_1^{-1}\left(1\right)\right) & 
- & h_1 h_2 \theta\left(\alpha_2^{-1}\left(1\right),\beta_2^{-1}\left(1\right)\right)   \\

 -  & h_1 h_2 \theta\left(\alpha_2^{-1} \alpha_1^{-1}\left(1\right),\beta_2^{-1} \beta_1^{-1}\left(y\right)\right)	& 
+ & h_1 h_2 \theta\left(\alpha_2^{-1} \alpha_1^{-1}\left(1\right),\beta_2^{-1}\left(1\right)\right) \\
 + & h_1 h_2 \theta\left(\alpha_2^{-1}\left(1\right),\beta_2^{-1} \beta_1^{-1}\left(y\right)\right) & 
- & h_1 h_2 \theta\left(\alpha_2^{-1}\left(1\right),\beta_2^{-1}\left(1\right)\right)   \\

 +  & h_1 h_2 \theta\left(\alpha_2^{-1} \alpha_1^{-1}\left(1\right),\beta_2^{-1} \beta_1^{-1}\left(1\right)\right)	& 
- & h_1 h_2 \theta\left(\alpha_2^{-1} \alpha_1^{-1}\left(1\right),\beta_2^{-1}\left(1\right)\right) \\
 - & h_1 h_2 \theta\left(\alpha_2^{-1}\left(1\right),\beta_2^{-1} \beta_1^{-1}\left(1\right)\right) & 
+ & h_1 h_2 \theta\left(\alpha_2^{-1}\left(1\right),\beta_2^{-1}\left(1\right)\right)   \\

\end{array} \] 
\end{small}
which becomes after cancellation:
\begin{small}
\[ \begin{array}{c}
h_1 h_2 \theta\left(\alpha_2^{-1} \alpha_1^{-1}\left(x\right),\beta_2^{-1} \beta_1^{-1}\left(y\right)\right)
-  h_1 h_2 \theta\left(\alpha_2^{-1} \alpha_1^{-1}\left(x\right),\beta_2^{-1} \beta_1^{-1}\left(1\right)\right)	\\
 -   h_1 h_2 \theta\left(\alpha_2^{-1} \alpha_1^{-1}\left(1\right),\beta_2^{-1} \beta_1^{-1}\left(y\right)\right)	 
 +   h_1 h_2 \theta\left(\alpha_2^{-1} \alpha_1^{-1}\left(1\right),\beta_2^{-1} \beta_1^{-1}\left(1\right)\right)	 \\
\end{array} \] 
\end{small}
We recognize $\left(^{(t_1t_2,h_1h_2)} \theta\right)(x,y)$, and we are done.
It is also easy to check that $^{(t,h)}(\theta_1 + \theta_2) =^{(t,h)} \theta_1 +^{(t,h)} \theta_2$.
\end{proof}

We provided this heavy computation to emphasize that, at this point, the reason why $N$ gives rise to an action of $\Atp(F,A)$ on $\B(F,A)$ seems to lie on a lucky coincidence.
$N$ is actually far more that just a naively-defined projection, and we will see in the proof of Theorem \ref{theorem} that it expresses well the relation between central extensions and their principal isotopes.

Moreover, it is easy to check that
\[ 
^{(t,h)}\widehat{\tau} = \widehat{\tau'}
\]
where $\tau' \in \Map_0$ is defined by 
\[
\tau' (x) = h \tau \gamma^{-1} (x) - h \tau \gamma^{-1} (1).
\]

Therefore, the action of $\Atp(F,A)$ on $\C(F,A)$ induces an action on $\B(F,A)$ and $\HH(F,A)$.

\mbox{}

The following lemma asserts that any orbit for the action of $\Atp(F,A)$ is constituted of loops with the same isotopism type.
\begin{lemma}
\label{isotopic}
For any $t = (\alpha, \beta, \gamma)\in\Atp(F)$, $h\in\Aut(A)$, the triple $\overline{t} = (\overline{\alpha},\overline{\beta},\overline{\gamma})$ defined by

\[
 \begin{cases}
\overline{\alpha}(x,a) = \left(\alpha (x), ha + h\theta(x, \beta^{-1}(1))\right) \\
\overline{\beta}(y, b) = \left(\beta (y), hb + h\theta(\alpha^{-1}(1), y)\right) \\
\overline{\gamma}(z, c) = \left(\gamma (z), hc + h\theta(\alpha^{-1}(1),\beta^{-1}(1))\right)
\end{cases}
\]

is an isotopism from $Q(F, A, \theta)$ to $Q(F, A, ^{(t,h)}\theta)$.
\end{lemma}
\begin{proof}
Let $\cdot_\theta$ be the multiplication in $Q(F, A, \theta)$ and $\cdot_{^{(t,h)}\theta}$ the multiplication in $Q(F, A, ^{(t,h)}\theta)$. Then
\begin{align*}
\overline{\alpha}(x,a) \cdot_{^{(t,h)}\theta} \overline{\beta}(y, b) = \; & \left(\alpha (x), ha + h\theta(x, \beta^{-1}(1))\right)\\
&\cdot_{^{(t,h)}\theta} \left(\beta (y), hb + h\theta(\alpha^{-1}(1), y)\right) \\
= \;&\big(\alpha (x) \beta (y), ha+hb+h\theta(x, \beta^{-1}(1))+ h\theta(\alpha^{-1}(1), y) \\ 
 &\;+ N(h\theta(\alpha^{-1},\beta^{-1}))(\alpha (x),\beta (y))\big) \\
=\; &\left(\gamma(xy), ha+hb + h \theta(x,y) +h\theta(\alpha^{-1}(1), \beta^{-1}(1))\right) \\
=\; &\overline{\gamma}(xy, a+b + \theta(x,y)) \\
=\; & \overline{\gamma}((x,a)\cdot_\theta (y,b)).
\end{align*}
\end{proof}

\section{Separability}
\label{sep}

As in \cite{Petr}, we define isotopy separability in the following way:

Write $\theta \sim \mu$ if $\mu =^{(t,h)} \theta + \tau$ for some $(t, h)\in\Atp(F,A)$, $\tau\in\B(F,A)$. 
$\sim$ is an equivalence relation on $\C(F,A)$, and by Lemmas \ref{isomorphic} and \ref{isotopic}, if $\theta \sim \tau$, then
$Q(F,A,\theta) \simeq Q(F,A,\mu)$. We say that $\theta$ is \textit{(isotopy) separable} if the converse also holds, i.e. if whenever $Q(F,A,\theta) \simeq Q(F,A,\mu)$ for some cocycle $\mu$, we also have $\theta \sim \mu$.

\begin{theorem}
\label{theorem}
Let $\theta \in \C(F,A)$. Set $Q_\theta=Q(F,A,\theta)$.
If $\Aut( Q_\theta)$ acts transitively on 
\[
\{K \leq Z( Q_\theta);\;K \cong A,\,  Q_\theta/K \simeq F \}
\]
then $\theta$ is isotopy separable.
\end{theorem}
\begin{proof}
Let $t=(\alpha,\beta,\gamma)$ be an isotopism between $ Q_\theta$ and $Q_\mu=Q(F,A,\mu)$, for some cocycle $\mu$.

The first step of the proof is to consider the splitting of $t$ into an isomorphism and a principal isotopism (i.e. an isotopism that has identity as its third component, see \cite{Pflug}).

Thus, let $(L,\ast)$ be the loop defined on $F \times A$ so that $\gamma$ is an isomorphism from
$ Q_\theta$ to $(L,\ast)$.
Then $(\overline{\alpha} = \alpha \gamma^{-1},\overline{\beta} = \beta \gamma^{-1}, \Id)$ is a principal isotopism between $L$ and $ Q_\mu$.

\begin{diagram}
 Q_\theta &\rTo^{(\gamma,\gamma,\gamma)} &(L,\ast)\\
\dTo^{(\alpha,\beta,\gamma)} &\ldTo_{(\overline{\alpha},\overline{\beta},\Id)} \\
 Q_\mu
\end{diagram}

We would like to understand the multiplication in $L$.

Let $e$ be the neutral of the loop $L$. 
Write $(x_0,a_0)=\overline{\beta}(e)$, $(y_0,b_0)=\overline{\alpha}(e)$.
$\overline{t}$ is a isotopism, thus 
\[
\overline{\alpha}(x,a)\cdot_\mu \overline{\beta}(y,b)=(x,a)\ast(y,b) .
\]
In particular,
\[
\begin{cases}
 \overline{\alpha}(x,a)\cdot_\mu (x_0,a_0)=(x,a)\ast e=(x,a) \\
(y,b)\cdot_\mu \overline{\beta}(y,b)=(y,b)\ast e=(y,b)
\end{cases}.
\]
We can invert this system to find
\[
\begin{cases}
 \overline{\alpha}(x,a)=(x/x_0,a-a_0-\mu(x/x_0,x_0))\\
 \overline{\beta}(y,b)=(y_0 \backslash y, b-b_0 -\mu(y_0,y_0 \backslash y))
\end{cases}.
\]

Therefore, the multiplication in $L$ is simply
\begin{align*}
(x,a)\ast(y,b) 	= & \overline{\alpha}(x,a)\cdot_\mu \overline{\beta}(y,b) \\
		= & \Big(x/x_0.y_0\backslash y, \, a + b - a_0-b_0 - \mu(x/x_0, x_0)-\mu(y_0, y_0\backslash y) \\
		  & + \mu(x/x_0, y_0\backslash y)\Big) .\\
\end{align*}

To put it in a more familiar form, let us write $(z_0,c_0)=e$. Now since
\[
\overline{\alpha}(e)\cdot_\mu \overline{\beta}(e)=e\ast e=e 
\]
i.e.
\[
 (y_0,b_0)\cdot_\mu (x_0,a_0)=(y_0x_0,a_0+b_0+\mu(y_0,x_0))=(z_0,c_0)
\]
we must have 
\[
 \begin{cases}
  y_0 x_0=z_0\\
  -a_0-b_0=\mu(y_0,x_0)-c_0
 \end{cases}.
\]

Thus the mutiplication in $L$ takes the form:

\[
(x, a) \ast (y, b) = (x/x_0.y_0\backslash y, a + b - c_0 + \widetilde{\mu}(x, y))
\]
for $\widetilde{\mu}$ defined by
\begin{small}
\[
\widetilde{\mu}(x, y) = \mu(x/x_0, y_0\backslash y) - \mu(x/x_0, y_0\backslash z_0)-\mu(z_0/x_0, y_0\backslash y)+\mu(z_0/x_0, y_0\backslash z_0).
\]
\end{small}

\mbox{}

The second step of the proof is now to recognize some subgroup of $ Q_\theta$ on which we can apply the hypothesis.

Notice that we always have
\[
(z_0, a + c_0) \ast (z_0, b + c_0) = (z_0, a + b + c_0).
\]
Thus the map $a\mapsto (z_0, a + c_0)$ is an isomorphism from $A$ onto
\[
K_0 = \{(z_0, a); a \in A\} ,
\]
$K_0$ being equipped with the multiplication $\ast$. 

Similarly, it is easy to check that $K_0 \leq Z(L)$. In particular, $L/K_0$ is a loop, and $F$ is isotopic to it via the triple of bijections $F\to L/K_0$:
\[
 \begin{cases}
  x \mapsto (x\,x_0,0)\ast K_0 \\
  y \mapsto (y_0\,y,0)\ast K_0 \\
  z \mapsto (z,0)\ast K_0
 \end{cases}.
\]

Therefore, $\gamma^{-1}$ being an isomorphism between $L$ and $ Q_\theta$ , we can apply the hypothesis to $\gamma^{-1}(K_0)$; thus there exists some automorphism $g$ of $Q_\theta$ such that $g(1\times A) =\gamma^{-1}(K_0)$. 
As a conclusion, precomposing with $g$ if necessary, we can always assume that
\[
\gamma(1 \times A) = K_0.
\]

\mbox{}

Now, what we have left to do is simply to express this fact with mappings. This is in direct analogy this \cite{Petr}.

Define a map $h:A\to A$ by
\[
\gamma(1, a) = (z_0, h(a) + c_0).
\]
Notice that 
\[
\gamma(1,a)\ast\gamma(1,b)=(z_0,h(a)+c_0)\ast(z_0,h(b)+c_0)=(z_0,h(a)+h(b)+c_0).
\]
Since $\gamma$ is an isomorphism between $ Q_\theta$ and $L$, this is also
\[
\gamma((1,a)\cdot_\theta(1,b))=\gamma(1,a+b)=(z_0,h(a+b)+c_0).
\]
Thus, $h \in \Aut(A)$.

Define also $k:F\to F$ and $\tau:F\to A$ by 
\[
\gamma(x, 0) = (k(x), \tau(x)+c_0).
\]
We have of course $\gamma(1, 0)=e=(z_0,c_0)$, so $k(1) = z_0$ and $\tau(1) = 0$; in particular $\tau \in \Map_0(F,A)$.

Moreover, computing in two ways $\gamma(xy,0)=\gamma(x,0)\ast \gamma(y,0)$ yields the following identity for $k$:
\[
 k(x)/x_0.y_0 \backslash k(y)=k(xy).
\]

We can now express $\gamma$ in term of these maps:
\begin{align*}
\gamma(z,c)=&\gamma((z,0)\cdot_\theta(1,c))=(k(z),\tau (z)+c_0)\ast(z_0,hz+c_0)\\
=&(k(z),hz+\tau (z)+c_0) .
\end{align*}

Recall that we also know the expression of $\overline{\alpha} = \alpha \gamma^{-1}$ and $\overline{\beta}= \beta \gamma^{-1}$, so by composition with $\gamma$, we get:
\[
\begin{cases}
 \alpha(x,a)=(k(x)/x_0,hx+\tau (x)+c_0-a_0-\mu(k(x)/x_0,x_0))  \\
  \beta(y,b)=(y_0 \backslash k(y),hy+\tau (y)+c_0-b_0-\mu(y_0,y_0 \backslash k(y)))
\end{cases}.
\]

After writing explicitly that $\alpha(x,a)\cdot_\mu\beta(y,b)$ is always equal to $\gamma((x,a)\cdot_\theta(y,b))$, we get
\[
 h\theta + \widehat{\tau} = N(\mu(\widetilde{\alpha}, \widetilde{\beta})) 
\]
where $\widetilde{t}=(\widetilde{\alpha}, \widetilde{\beta}, \widetilde{\gamma})$ is defined to be the triple 
\[
 \begin{cases}
  \widetilde{\alpha}(x)=k(x)/x_0 \\
  \widetilde{\beta}(y)=y_0\backslash k(y)\\
  \widetilde{\gamma}(z)=k(z)
 \end{cases}.
\]

Now $\widetilde{t} \in \Atp(F)$, $h\in \Aut(A)$ and $\tau \in \Map_0(F,A)$, so $\theta \sim \mu$.

Thus $\theta$ is separable.
\end{proof}

We leave to the reader to check that the following results, proved in \cite{Petr}, 3.3-3.7, still hold in our setting, thanks to Theorem \ref{theorem} (we recall that if a loop is isotopic to a group, then it is isomorphic to it, see \cite{Pflug}).
 
\begin{proposition}
If $Q(F,A,\theta )$ is an abelian group, and $A = \mathbb{Z}_p$ for $p$ a prime integer, then $\theta$ is isotopy separable.
\end{proposition}

\begin{lemma}
Let $Q =  Q_\theta$, $A = \mathbb{Z}_p$, $p$ a prime. Assume further that one of the following conditions is satisfied:
\begin{enumerate}[(i)]
\item $|Q| = p$,
\item $|Q| = pq$, where $q$ is a prime,
\item $[Q : Z(Q)] \leq 2$,
\item $|Q| < 12$.
\end{enumerate}
Then $\theta$ is isotopy separable.
\end{lemma}


\section{The invariant subspaces}
\label{inv}

Following \cite{Petr}, define for $(t,h)\in\Atp(F,A)$:
\[
\Inv(t,h)= \{\theta \in \C(F,A);\; \theta -^{(t,h)}\theta \in \B(F,A)\}
\]
and for $\emptyset \neq H \subset \Atp(F,A)$:
\[
\Inv(H)= \bigcap_{(t,h)\in H}\Inv(t,h).
\]
We state the following, the proof of which is exactly the same as in \cite{Petr}:
\begin{lemma}
Let $\emptyset \neq H \subset \Atp(F,A)$. Then 
\[
\Inv(H) = \Inv(\langle H\rangle ).
\]
\end{lemma}
\begin{corollary}
Let $H, K \leq \Atp(F,A)$. Then 
\[
\Inv(H)\cap \Inv(K)=\Inv(\langle H\cup K\rangle).
\]
\end{corollary}
For $t,u\in \Atp(F)$ and $h,k\in \Aut(A)$, let $^u t=utu^{-1}$, $^k h=khk^{-1}$. 
\begin{lemma}
Let $(t,h),\,(u,k)\in\Atp(F,A)$. Then 
\[
\theta \in \Inv(t,h) \text{ if and only if } ^{(u,k)}\theta \in \Inv(^ut,^kh).
\]
\end{lemma}

For $H \leq \Atp(F,A)$, let
\begin{align*}
&\Inv^\ast(H) = \{\theta \in \C(F,A);\; \theta \in \Inv(t,h) \text{ if and only if } (t,h)\in H\} ,\\
&\Inv^\ast_c(H)=\bigcup_{(t,h)\in\Atp(F,A)} \Inv^\ast(^{(t,h)}H).
\end{align*}

If $G$ is a group and $H \leq G$, let $N_G(H) = \{a \in G;\; ^aH = H\}$ be the normalizer of $H$ in $G$.
\begin{lemma}
Let $H \leq G = \Atp(F,A)$. Then 
\[
|\Inv^\ast_c(H)|=|\Inv^\ast(H)|\cdot [G : N_G(H)].
\]
\end{lemma}
For a group $G$, denote by $\Sub_c(G)$ a set of subgroups of $G$ such that for every $H \leq G$ there is precisely one $K \in \Sub_c(G)$ such that $K$ is conjugate to $H$.
\begin{theorem}
\label{theonum}
Let $F$ be a loop and $A$ an abelian group. Assume that $\theta$ is separable for
every $\theta \in \C(F,A)$. Let $G = \Atp(F,A)$. Then there are
\[
\sum_{H\in \Sub_c(G)} \frac{|\Inv^\ast_c(H)|}{|\B(F,A)|\cdot[G:H]}= \sum_{H\in \Sub_c(G)} \frac{|\Inv^\ast(H)|}{|\B(F,A)|\cdot[N_G(H):H]}
\]
central extensions of $A$ by $F$, up to isotopism.
\end{theorem}


\section{Nilpotent loops of order $2q$, $q$ prime}
\label{2q}
We now investigate the $ 2q $ order case, with $ q $ an odd prime integer throughout.
The discussion in \cite{Petr} showing that we can suppose $ A=\bbZ_2 $, $ F=\bbZ_q $ and that each cocycle is admissible is still valid; we can therefore use fully Theorem \ref{theonum} in the computation of the number of nilpotent loops of order $ 2q $. 
In order to do so, the first step is to understand the structure of $\Atp(F)$.

\subsection{Subgroup structure of $ \Atp(\bbZ_q) $}
\label{SS}

We recall the following proposition from \cite{Cla}.
\begin{proposition}
\label{proposition}
Let $G$ be a finite abelian group. Then
\begin{align*}
\phi: \Aut(G) \ltimes G^2 &\rightarrow \Atp(G) \\
(h,x_0,y_0) &\mapsto t_{h,x_0,y_0}
\end{align*}
is an isomorphism, where the multiplication on $ \Aut(G) \ltimes G^2$ is given by
\[
(h,X)(h',X')=(hh',hX'+X).
\]
and where the autotopisms $t_{h,x_0,y_0}$ are defined by
\[
\begin{cases}
	x \mapsto hx+x_0 \\
	y \mapsto hy+y_0\\
	z \mapsto hz+x_0+y_0
\end{cases} .
\]

\end{proposition}

Let us introduce some notation. For $m$ a generator of $F\setminus\{0\}\cong \bbZ_{q-1}$, $d$ a divisor of $q-1$, $X\in F^2$ and $y\in F$, define

\[
\begin{cases}
 H_d^X=\langle(m^d,X)\rangle=\{ (m^{kd},\frac{1-m^{kd}}{1-m^d} X);\; k \in \bbZ \} \\
 K_y=\langle(1,(1,y))\rangle=\{(1,(k,ky));\; k\in \bbZ \} \\
 \widetilde{K}=\langle(1,(0,1))\rangle=\{(1,(0,k));\; k\in \bbZ \}
\end{cases} .
\]

Since by \cite{Cla} for a fixed $d$ all $H_d^X$ are conjugate (see Table 1), we simply write $H_d$ instead of $H_d^{(0,0)}$. Note that this notation is consistent with the one in \cite{Petr}.

Here are now all subgroups of $\Atp(F)$, up to conjugacy

\begin{table}[H]
\begin{center}
\begin{tabular}{|l|l|l|l|}
\hline
subgroup $H$ & normalizer $N_G(H)$ & conjugates & $[N_G(H):H]$ \\ \hline\hline
$\{1\}$  & $\Atp(F)$ & only itself & $q^2(q-1)$ \\ \hline
$H_d,\, d\neq q-1$ & $\Aut(F)$ & every  $H_d^X$ &$d$ \\ \hline
$K_y$ or $\widetilde{K}$ &  $\Atp(F)$ & only itself & $q(q-1)$ \\  \hline
$H_d \cdot K_y$, $d\neq q-1$ & $\Aut(F)\cdot K_y$ &every $H_d^X\cdot K_y$ &$d$ \\  \hline
$H_d \cdot \widetilde{K}$, $d\neq q-1$ & $\Aut(F)\cdot \widetilde{K}$&every $H_d^X\cdot \widetilde{K}$  &$d$\\  \hline
$H_d\ltimes F^2$  & $\Atp(F)$ & only itself  &$d$\\ \hline
\end{tabular}
\end{center}
\caption{Representatives for conjugacy classes of $F=\Atp(\bbZ_q)$ and their normalizer.}
\end{table}
\begin{proof}
 See \cite{Cla}, Example 3.4.
\end{proof}

\subsection{$ \dim(\Inv(H))  $, $ H \leq  \Atp(\bbZ_q) $}

In the next proposition, we compute the dimensions of the invariant spaces of the subgroups of $\Atp(F)$, with as before $ A=\bbZ_2 $, $ F=\bbZ_q $ and $q$ an odd prime (see Subsection \ref{SS} for notations). 

\begin{proposition}
\label{p1}
The dimensions of the invariant spaces of the subgroups of $\Atp(F)$ are indicated in Table 2 below, where $d$ is any divisor of $q-1$.

\begin{table}[H]
\begin{center}
\begin{tabular}{|l||l|l|l|}
\hline
subgroup $H$ & $H_d$&$H_d\cdot K_y$, $y\notin \{0,-1\}$& other \\ \hline
$ \dim(\Inv(H)/\B(F,A))  $& $(q-2)d$&$d$&0\\ \hline
\end{tabular}
\end{center}
\caption{Representatives for conjugacy classes of $F=\Atp(\bbZ_q)$ and dimension of their invariant subspaces.}
\end{table}

\end{proposition}
\begin{proof}
The proof will take us the entire subsection, and will be divided in lemmas and corollaries as much as possible.

Note that since the action of $\Atp(F,A)$ we defined on $\C(F,A)$ coincides (by restriction) with the action of $\Aut(F,A)$ defined in \cite{Petr}, the first column of Table 2 directly follows from \cite{Petr}. Thus, let us start with the case $H=K_y$.

For every $y_0\in F$, define on $\C(F,A)$ the operator $S$ (depending on $y_0$) by:

\begin{align*}
S: \C(F,A) &\rightarrow \C(F,A) \\
\theta &\mapsto ^{(1,t_{1,1,y_0})}\theta -\theta
\end{align*}
using the notation of Proposition \ref{proposition};
otherwise put, $S$ is defined for every $\theta\in\C(F,A)$ by
\[
 S\theta (x,y)=\theta(x+1,y+y_0)-\theta(x+1,y_0)-\theta(1,y+y_0)+\theta(1,y_0)-\theta(x,y)
\]

Similarly, define on the space $\Map(F\times F,A)$ of non-normalized cocycles the operator $\widetilde{S}$ by:

\begin{align*}
\widetilde{S}: \Map(F\times F,A) &\rightarrow \Map(F\times F,A) \\
\mu &\mapsto \mu(\cdot +1,\cdot +y_0) -\mu
\end{align*}

i.e. for every $\mu\in\Map(F\times F,A)$:
\[
 \widetilde{S}\mu (x,y)=\mu(x+1,y+y_0)-\mu(x,y)
\]

Like in \cite{Petr}, since $\Inv(K_{y_0})=S^{-1}(\B(F,A))$, we are interested in computing the kernel $\Ker S$ first. In analogy with \cite{Petr}, we are going to prove that it is spanned by these cocycles $\Lambda_i$ that take the value 1 on exactly one orbit of the action on $F^2$ by the translation $(x,y)\mapsto (x+1,y+y_0)$; or rather by their \textit{image} $N(\Lambda_i)$ under $N$ (this is the content of Corollary \ref{c1}).

Namely, for $0\leq i\leq q-1$, define $\Lambda_i\in \Map(F\times F,A)$ by
\[ \Lambda_i(k,ly_0)=\delta_{l-k,i}=\begin{cases}
1 \text{ if } l-k=i \mod q\\
 0 \text{ otherwise}
\end{cases}.
\]

Note that these span $\Ker\widetilde{S}$. Also, 
\[
\Ker S =N(\Ker\widetilde{S}+V) 
\]
where $V$ is some vector space spanned by particular solutions to the systems 
\[
\widetilde{S}\mu=\nu 
\]
for every $\nu$ in a chosen basis of $\Ker N$.

\begin{lemma}
\label{l1}
 For any $y_0\in F$, we can choose $V$ so that $V\subset \Ker N$.
\end{lemma}
\begin{proof}
We have to separate two cases.

Suppose first that $y_0\neq 0$.
For $0\leq i,j\leq q-1$, define $L_i,C_j$ by
\[
 \begin{cases}
  L_i(x,y)=\delta_{x,i}\\
  C_j(x,y)=\delta_{y,j}
 \end{cases}.
\]
Note that these elements of $\Map(F\times F,A)$ are in $\Ker N$; write $1=\sum_i L_i=\sum_j C_j$ for the constant map equalling 1 everywhere.
Now, $\Ker N$ is easily seen to have dimension $2q-1$, with basis for instance 
\[
\{1, L_1, \ldots, L_{q-1}, C_1,\ldots, C_{q-1} \}
\]
or, better,
\[
\{1, L_0-L_1, \ldots, L_{q-2}-L_{q-1}, C_0-C_{y_0},\ldots, C_{(q-2)y_0}-C_{(q-1)y_0} \}.
\]
Therefore, we can choose $L_{i+1}$ (resp. $C_{(j+1)y_0}$), with $0\leq i,j\leq q-2$ as solutions to
\[
 \widetilde{S}\mu=L_i-L_{i+1} \text{  (resp. } C_{jy_0}-C_{(j+1)y_0} \text{)}
\]
and $V$ has dimension at least $2(q-1)$. Let us show that it cannot be more, by showing that the constant map 1 does not have any solution in $\Map(F\times F,A)$.

Indeed, if it were the case, an easy induction for such a solution $\mu$ would imply that for every integer $k\geq 1$
\[
 \mu(k,ky_0)=\mu(0,0)+k.
\]
In particular for $k=q$,
\[
\mu(0,0)=\mu(q,qy_0)=\mu(0,0)+q=\mu(0,0)+1 
\]
This is absurd, so $V$ has dimension $2(q-1)$, and can be chosen to be included in $\Ker N$.

Now, assume $y_0 =0$.
This case is similar, but here no $C_j$ for $0\leq j\leq q-1$ has a solution in $\Map(F\times F,A)$. 
Indeed, were it the case, 
\[
\mu(k,j)=\mu(0,j)+k 
\]
would hold for every integer $k\geq 1$; taking $k=q$, we would have $\mu(q,j)=\mu(0,j)+1 $, absurd.
Thus we can choose $V=\Span_{1\leq i\leq q-1}(L_i)$, and we are done.
\end{proof}

\begin{lemma}
\label{l2}
For any $y_0\neq 0$, $\Ker\widetilde{S}\cap\Ker N=\Span(1)$.
\end{lemma}
\begin{proof}
Suppose we have some $\mu \in\Ker\widetilde{S}\cap\Ker N$. Then for every integers $k,l\geq 1$, we have
\[
 \mu(k+1,(l+1)y_0)=\mu(k+1,0)+\mu(0,(l+1)y_0)-\mu(0,0).
\]
But this is also
\[
 \mu(k,ly_0)=\mu(k,0)+\mu(0,ly_0)-\mu(0,0)
\]
Thus $\mu(k+1,0)-\mu(k,0)$ does not depend on $k$, i.e.
\[
 \mu(k+1,0)=\mu(k,0)+c
\]
for some constant $c\in A$. Then by a quick induction
\[
 \mu(0,0)=\mu(q,0)=\mu(0,0)+qc=\mu(0,0)+c
\]
so $c=0$. Therefore $\mu(k+1,0)=\mu(k,0)$ for all $k$.

Similarly, $\mu(0,(l+1)y_0)=\mu(0,ly_0)$ for all $l$. But then $\mu$ must be constant, and we are done.
\end{proof}

\begin{corollary}
\label{c1}
If $y_0=0$, then $\Ker S=0$. Else, $\Ker S$ has dimension $q-1$ and basis $\{ N \Lambda_i\}_{1\leq i \leq q-1}$.
\end{corollary}
\begin{proof}
This is a direct corollary of Lemmas \ref{l1} and \ref{l2}. 
\end{proof}

The last step is now to compute the intersection $\Ker S \cap B$.

\begin{lemma}
If $y_0=-1$ then $\Ker S \subset B$. Else, $\Ker S \cap B = 0$.
\end{lemma}
\begin{proof}
In this proof, we use $A=\bbZ_2$ without warning. For convenience, we also define $z_0=y_0+1$.

First, if $y_0=-1$, every $\Lambda_i$ is in $B$. Thus, let us suppose $y_0$ is neither 0 nor $-1$, and take some 
\[
 \widehat{\tau} =  \displaystyle\sum_{c\neq 0} \lambda_c \widehat{\tau_c}
\]
that verifies $S\widehat{\tau}=0$, where as in \cite{Petr} we define every $\tau_c$ by
\[
 \tau_c(x)=\delta_{x,c}
\]

Since 
\[
S\widehat{\tau_c}=\begin{cases}
		\widehat{\tau_{c-z_0}}+\widehat{\tau_c} \text{ if } c\neq z_0\\
		\sum_{c'\neq 0,\, c'\neq z_0}\widehat{\tau_c'} \text{ otherwise}
\end{cases}
\]
we have 
\begin{align*}
S \widehat{\tau}= & \sum_{c\neq 0,\, c\neq z_0}\lambda_c (\widehat{\tau_{c-z_0}}+\widehat{\tau_c})\,+\,\lambda_{z_0}\cdot\sum_{c\neq 0,\, c\neq z_0}\widehat{\tau_c} \\
		= & \sum_{c\neq 0,\, c\neq z_0,\, c\neq -z_0}(\lambda_{c+z_0}+\lambda_c+\lambda_{z_0})\widehat{\tau_c} \\
		&+\lambda_{2z_0}\widehat{\tau_{z_0}} +(\lambda_{z_0}+\lambda_{-z_0})\widehat{\tau_{-z_0}}.
\end{align*}
Because the $\tau_c$ for $c\neq 0$ form a basis of $\B(F,A)$, we must conclude that 
\begin{align*}
\lambda_{2z_0}&=0=2\lambda_{z_0}\\
\lambda_{3z_0}&=\lambda_{2z_0}+\lambda_{z_0}=3\lambda_{z_0}\\
&\ldots\\
\lambda_{(q-1)z_0}&=(q-1)\lambda_{z_0}=0\\
\lambda_{-z_0}&=\lambda_{z_0}.
\end{align*}
Thus $\lambda_{z_0}=0$, so $\lambda_{kz_0}=0$ for every $k$, hence $\widehat{\tau}=0$.
\end{proof}

As a quick corollary, we are done for the second column of Table 2, in the case $d=q-1$:
\begin{corollary}$\dim(\Inv(K_{y})/\B(F,A))=q-1$
whenever $y\notin \{0,-1\}$. Moreover, the invariant spaces of $ K_0$, $K_{-1}$ and $\widetilde{K}$ are null mod $\B(F,A)$. 
\end{corollary}
\begin{proof}
 The only case that was not already investigated is $H=\widetilde{K}$, but this is symmetric to the case $H=K_0$.
\end{proof}

Note that any subgroup $H$ in the third column of Table 2 has either $K_0$, $K_{-1}$ or $\widetilde{K}$ as a subgroup. Thus, its invariant space is also null mod $\B(F,A)$.

The only remaining cases in Proposition \ref{p1} are $H=H_d\cdot K_y$, for $y\notin \{0,-1\}$ and $d\neq q-1$.
Start with a cocycle $\theta\in \Span_{1\leq i\leq q-1} (N\Lambda_i)\oplus \B(F,A)$
\[
 \theta=\sum_{i\neq0}\lambda_i N\Lambda_i \,+\, \tau.
\]

Then $^{(h,(0,0))}\theta-\theta \in \B(F,A)$ if and only if 
\[
 \sum_{i\neq0}\lambda_i (N\Lambda_{hi}-N\Lambda_i)\in \B(F,A)
\]
but since the $\Lambda_i$ are linearly independant over $\Ker N$, this is equivallent to
\[
  \sum_{i\neq0}\lambda_i (\Lambda_{hi}-\Lambda_i)=0
\]
i.e.
\[
  \sum_{i\neq0}\Lambda_i (\lambda_{h^{-1}i}-\lambda_i)=0
\]
i.e $\lambda_i=\lambda_{hi}$ for all $i$.
Thus for any $y\notin \{0,-1\}$,
\[
 \dim\Big(\Inv\big(\{(h,(0,0))\}\cup K_{y}\big)/\B(F,A)\Big)=\frac{q-1}{|h|}
\]
and all the cases in Proposition \ref{p1} are covered.
\end{proof}

\subsection{$  | \Inv^\ast(H) |  $, $ H \leq  \Atp(\bbZ_q) $ and $\widetilde{\mathcal{N}}(2q)$}

Before computing the number of nilpotent loops of order $2q$ up to isotopism, we still have to compute the cardinalities of the starred invariant spaces for the subgroups of $\Atp(F,A)$. This is the content of Proposition \ref{l3}.

\begin{proposition}
 \label{l3}
The cardinalities of the starred invariant spaces for the subgroups of $\Atp(F,A)$ are provided in Table 3 below, where as in \cite{Petr}, we define for every integer d:
\[
 \Pred(d)=\{d';\; 1\leq d'<d,\, d/{d'} \text{ is a prime}\}.
\]

\begin{table}[H]
\begin{center}
\begin{tabular}{|l||l|}
\hline
	subgroup $H$ & cardinality $| \Inv^\ast(H) |$ \\ \hline\hline
	$\{1\}$  & $2^{(q-2)(q-1)}+q^2\cdot\displaystyle\sum_{D\subset \Pred(q-1)} (-1)^{|D|} 2^{(q-2)\gcd(D)} $\\
&$-(q-2)\Big(2^{q-1}+q^2\cdot\displaystyle\sum_{D\subset \Pred(q-1)} (-1)^{|D|} 2^{\gcd(D)}  \Big)$\\
&$-(q-3)(q-1)(q+1)$\\ \hline
	$H_d,\, d\notin\{1, q-1\}$ & $2^{(q-2)d}+\displaystyle\sum_{D\subset \Pred(d)} (-1)^{|D|} 2^{(q-2)\gcd(D)} $\\
&$-(q-2)\Big(2^d+\displaystyle\sum_{D\subset \Pred(d)} (-1)^{|D|} 2^{\gcd(D)}  \Big)$ \\ \hline
	$H_1$ & $2^{q-2}-(q-1) $\\ \hline
	$K_y, \, y \notin \{0,-1 \}$ &$ 2^{q-1}+q\cdot\displaystyle\sum_{D\subset \Pred(q-1)} (-1)^{|D|} 2^{\gcd(D)} +q-1$  \\  \hline
	$H_d \cdot K_y, d \notin\{1, q-1\}$ & $2^d+\displaystyle\sum_{D\subset \Pred(d)} (-1)^{|D|} 2^{\gcd(D)}$ \\
$y\notin\{0,-1\}$ & \\  \hline
$H_1\cdot K_y, \, y \notin \{0,-1 \}$&$1$\\ \hline
	$\Atp(F,A)$&1\\ \hline
	other  & 0\\ \hline
\end{tabular}
\end{center}
\caption{Representatives for conjugacy classes of $F=\Atp(\bbZ_q)$ and their starred invariant spaces.}
\end{table}
\end{proposition}

\begin{proof}
The proof is straightforward, using the following expression, together with Proposition \ref{p1} and a standard inclusion/exclusion argument. 
\begin{align*}
 \Inv^\ast (H)	=&\Inv(H)\setminus \displaystyle\cup_K\Inv(K) \\
		=&(\Inv(H)\setminus \{0\})\setminus (\displaystyle\cup_K\Inv(K)\setminus \{0\})
\end{align*}
where the union is taken for subgroups $K$ such that $H$ is a maximal subgroup of $K$;
Table \ref{t4} below provides for each subgroup $H$ the subgroups $K$ in which $H$ is maximal.

\begin{table}[H]
\label{t4}
\begin{center}
\begin{tabular}{|l||l|}
\hline
	subgroup $H$ & subgroups $K$ in which $H$ is maximal \\ \hline\hline
	$\{1\}$  & every $H_d^X$ for $d\in \Pred(q-1)$, any $X$ \\
		& every $K_y$ for $y\in \{1,\ldots ,q-2\}$\\ \hline
	$H_d,\, d\neq q-1$ & every $H_{d'}$ for $d'\in \Pred(d)$, \\
		& every $H_d\cdot K_y$ for $y\in \{1,\ldots ,q-2\}$\\ \hline
	$K_y, \,y\in \{1,\ldots ,q-2\}$ &every $H_d^X\cdot K_y$ for $d\in \Pred(q-1)$ and \\
					& $X\in \{ (0,0),\ldots (q-1,0)\}$\\  \hline
	$H_d \cdot K_y,\,y\in \{1,\ldots ,q-2\}$ & every $H_{d'}\cdot K_y$ for $d'\in \Pred(d)$ and \\& $y\in \{1,\ldots ,q-2\}$\\  \hline
\end{tabular}
\end{center}
\caption{Representatives for conjugacy classes of $F=\Atp(\bbZ_q)$ and the non-null invariant-space subgroups in which they are maximal.}
\end{table}

Details are left to the reader. 
\end{proof}

For convenience, let us write  $ \widetilde{\mathcal{N}}(n)$ for the number of nilpotent loops of order $n$ counted up to isotopism, and  $\mathcal{N}(n)$ the number of nilpotent loops of order $n$ counted up to isomorphism. This notation is consistent with the one in \cite{Petr}, and we recall the following:

\begin{theorem}
 Let $q$ be an odd prime. Then the number $\mathcal{N}(2q)$ of nilpotent loops of order $2q$ counted up to isomorphism is
\[
 \mathcal{N}(2q)=\displaystyle\sum_{d \text{ divides } q-1} \frac{1}{d} \Bigg(2^{(q-2)d}+\displaystyle\sum_{D\subset \Pred(d)} (-1)^{|D|} 2^{(q-2)\gcd(D)} \Bigg).
\]
\end{theorem}
\begin{proof}
See \cite{Petr}, Theorem 7.1.
\end{proof}

We have now all ingredients in hand for Theorem \ref{lastthm}.

\begin{theorem}
\label{lastthm}
Let $q$ be an odd prime. Then the number $ \widetilde{\mathcal{N}}(2q)$ of nilpotent loops of order $2q$ counted up to isotopism is
\begin{align*}
  \widetilde{\mathcal{N}}(2q)=	& \frac{2^{(q-2)(q-1)}}{q^2(q-1)}+ \frac{1}{q-1} \displaystyle\sum_{D\subset \Pred(q-1)} (-1)^{|D|} 2^{(q-2)\gcd(D)} \\
				& +\displaystyle\sum_{d \text{ strictly divides } q-1} \frac{1}{d} \Bigg(2^{(q-2)d}+\displaystyle\sum_{D\subset \Pred(d)} (-1)^{|D|} 2^{(q-2)\gcd(D)} \Bigg) \\
				& +\frac{1}{q^2}\Big((q-2)2^{q-1}+3)\Big)  \\
= & \mathcal{N}(2q) + \frac{1}{q^2}\big(-(q+1)2^{(q-2)(q-1)}+(q-2)2^{q-1}+3\big) .
\end{align*}

\end{theorem}
\begin{proof}
Combine Theorem 5.5 and Proposition 6.8.
\end{proof}

Recall from \cite{Petr} the following theorem.
\begin{theorem}
 \label{lastpetr}

 Let $q$ be an odd prime. Then the number of nilpotent loops of order $2q$ counted up to isomorphism is approximately $2^{(q-2)(q-1)}/{(q-1)}$. More precisely,

\[  \displaystyle\lim_{q \text{ prime, } q\to \infty}\mathcal{N}(2q)\cdot \frac{q-1}{2^{(q-2)(q-1)}}=1.
\]

\end{theorem}
\begin{proof}
See \cite{Petr}, Theorem 7.3.
\end{proof}

We can now compare the estimates for $\mathcal{N}(2q)$ and $\widetilde{\mathcal{N}}(2q)$, this is the purpose of the following corollary.

\begin{corollary}
 Let $q$ be an odd prime. Then the number of nilpotent loops of order $2q$ counted up to isotopism is approximately $2^{(q-2)(q-1)}/{q^2(q-1)}$. Thus, the ratio between the number of such loops counted up to isomorphism and up to isotopism is approximately $q^2$. More precisely,

\begin{align*}
  \displaystyle\lim_{q \text{ prime, } q\to \infty} \widetilde{\mathcal{N}}(2q)\cdot \frac{q^2(q-1)}{2^{(q-2)(q-1)}}=1, \\
 \displaystyle\lim_{q \text{ prime, } q\to \infty} \frac{\mathcal{N}(2q)}{q^2\cdot\widetilde{\mathcal{N}}(2q)}=1 .
\end{align*}
\end{corollary}
\begin{proof}
 This is immediate from Theorems \ref{lastthm} and \ref{lastpetr}.
\end{proof}

Table \ref{tablelast} below provides $\widetilde{\mathcal{N}}(2q)$ for any odd prime $q\leq 17$. Like in \cite{Petr}, it is not a problem to compute $\widetilde{\mathcal{N}}(2q)$ for bigger primes, but this would not fit nicely in a table.

\begin{table}[H]
\label{tablelast}
\begin{center}
\begin{tiny}
\begin{tabular}{l|l}
$q$ & $\widetilde{\mathcal{N}}(2q)$ \\ \hline 
$3$ & $\np{2}$ \\ 
$5$ & $\np{63}$ \\ 
$7$ & $\np{3658003}$ \\ 
$11$ & $\np{1023090941561683953759579}$ \\ 
$13$ & $\np{2684673506279593406254437209960379083}$ \\ 
$17$ & $\np{382103603974564085117495134243710834769544696954218618882023686506659}$ \\ 
\end{tabular}
\end{tiny}
\end{center}
\caption{Number $\widetilde{\mathcal{N}}(2q)$ of nilpotent loops of order $2q$ up to isotopism, for odd primes $q\leq 17$.}
\end{table}

\section{Conclusion}
\label{concl}
We invite the reader desiring to know about related works and topics to check Section 10 in \cite{Petr}.

Note that in the present paper we did not compute the number of nilpotent loops of small order (say less that 24) up to isotopy.
Undertaking such counting appears of interest to us.
Possible trouble could be the isotopy non-invariance of the set of large center cocycles (see Section 8 in \cite{Petr}), since isotopy does not preserve centers.

Also of interest is the enumeration of nilpotent loops of small order in Bol-Moufang varieties (see \cite{BolMouf}) up to isomorphy, and up to isotopy (here also, isotopy invariance should be a concern).

The computation of Table \ref{tablelast} was undertaken using the GAP System for Computational Discrete Algebra (see http://www.gap-system.org/).
This paper comes with the code used for Table \ref{tablelast} and a file containing the numbers $\widetilde{\mathcal{N}}(2q)$ of nilpotent loops of order $2q$ for every odd prime $q$ less than 100.
The two files can be downloaded at http://www.math.cornell.edu/\string~lpc49/.

\bibliography{biblio} 
\bibliographystyle{alpha}

\end{document}